\documentclass[12pt]{amsart}
\usepackage{amsthm,amsopn,amsfonts,amssymb,pb-diagram}
\usepackage{hyperref,cancel}
\usepackage{tikz} 
\usepackage{fancyhdr}
\usepackage{young}

\usepackage{euler, amsfonts, amssymb, latexsym, epsfig, epic}
\usepackage{epstopdf}

\font\co=lcircle10
\def\petit#1{{\scriptstyle #1}}

\def\jr{\smash{\raise2pt\hbox{\co \rlap{\rlap{\char'005} \char'007}}
               \raise6pt\hbox{\rlap{\vrule height6.5pt}}
               \raise2pt\hbox{\rlap{\hskip4pt \vrule height0.4pt depth0pt
                width7.7pt}}}}
\def\je{\smash{\raise2pt\hbox{\co \rlap{\rlap{\char'005}
                \phantom{\char'007}}}\raise6pt\hbox{\rlap{\vrule height6pt}}}}
\def\+{\smash{\lower2pt\hbox{\rlap{\vrule height14pt}}
                \raise2pt\hbox{\rlap{\hskip-3pt \vrule height.4pt depth0pt
                width14.7pt}}}}
\def\perm#1#2{\hbox{\rlap{$\petit {#1}_{\scriptscriptstyle #2}$}}%
                \phantom{\petit 1}}

\def\textcross{\ \smash{\lower4pt\hbox{\rlap{\hskip4.15pt\vrule height14pt}}
                \raise2.8pt\hbox{\rlap{\hskip-3pt \vrule height.4pt depth0pt
                width14.7pt}}}\hskip12.7pt}
\def\textelbow{\ \hskip.1pt\smash{\raise2.8pt%
                \hbox{\co \hskip 4.15pt\rlap{\rlap{\char'005} \char'007}
                \lower6.8pt\rlap{\vrule height3.5pt}
                \raise3.6pt\rlap{\vrule height3.5pt}}
                \raise2.8pt\hbox{%
                  \rlap{\hskip-7.15pt \vrule height.4pt depth0pt width3.5pt}%
                  \rlap{\hskip4.05pt \vrule height.4pt depth0pt width3.5pt}}}
                \hskip8.7pt}

\newtheorem{Theorem}{Theorem}[section]

\newtheorem*{Theorem*}{Theorem}
\newtheorem*{Lemma*}{Lemma}
\newtheorem{Lemma}[Theorem]{Lemma}

\newtheorem{Proposition}[Theorem]{Proposition}

\theoremstyle{definition}

\theoremstyle{remark}

\newcommand{\from}{\leftarrow}
\newcommand{\onto}{\twoheadrightarrow}

\newcommand{\into}{\hookrightarrow}

\newcommand{\CC}{\mathbb{C}}

\newcommand{\ZZ}{\mathbb{Z}}
\newcommand{\PP}{\mathbb{P}}

\newcommand\defn[1]{{\bf #1}}

\newcommand\iso{\cong}

\newcommand\barX{{\overline X}}

\newcommand\calD{{\mathcal D}}
\newcommand\naturals{{\mathbb N}}

\newcommand\union{\cup}
\newcommand\Union{\bigcup}

\DeclareMathOperator{\rank}{rank}

\newcommand\junk[1]{}

\begin{document}

\title[Schubert polynomials, pipe dreams, equivariant classes,
and a co-transition formula]
{Schubert polynomials, pipe dreams, equivariant classes,\\ 
  and a co-transition formula}

\author{Allen Knutson}
\address{Department of Mathematics, Cornell University, Ithaca, NY 14853 USA}
\email{allenk@math.cornell.edu}
\thanks{AK was partially supported by NSF grant DMS-0902296.}
\date{July 29, 2020}

\begin{abstract}
  We give a new proof that three families of polynomials coincide:
  the double Schubert polynomials of Lascoux and Sch\"utzenberger
  defined by divided difference operators, the pipe dream polynomials
  of Bergeron and Billey, and the equivariant cohomology classes of
  matrix Schubert varieties. All three families are shown to satisfy a
  ``co-transition formula'' which we explain to be some extent
  projectively dual to Lascoux' transition formula.
  We comment on the $K$-theoretic extensions.
\end{abstract}

\dedicatory{For Bill Fulton's 80th birthday}
\maketitle

\setcounter{tocdepth}{1}
{\footnotesize 
\tableofcontents}


\newcommand\Ssub[1]{{S_{#1}}}
\newcommand\Sn{\Ssub n}
\newcommand\Sinf{\Ssub \naturals}
\newcommand\NN\naturals

\section{Overview}\label{sec:overview}

Let $S_\infty := \union_{n=1}^\infty S_n$ be the permutations $\pi$ of 
$\NN_+$ that are eventually the identity, i.e. $\pi(i) = i$ for $i\gg 0$.
We define three families of polynomials in $\ZZ[x_1,x_2,\ldots,y_1,y_2,\ldots]$,
named $A$(lgebra), $C$(ombinatorics), and $G$(eometry), 
and each indexed by $S_\infty$:

\begin{enumerate}
\item {\em Double Schubert polynomials $A_\pi$.} These were defined by
  Lascoux and Sch\"utzenberger \cite{LascouxSchutzenberger}, 
  using a recurrence relation based on
  divided difference operators. We recapitulate the definition in 
  \S \ref{sec:schubert}, with a mildly novel approach.
\item {\em Pipe dream polynomials $C_\pi$.} These were introduced (in the
  $(x_i)$ variables only, and not called this) by N. Bergeron and 
  Billey \cite{BB}; we recall them in \S \ref{sec:pipedream}.
\item {\em Matrix Schubert classes $G_\pi$.} These were introduced 
  by Fulton \cite{Fulton92, universal} (and again, not called this)
  to give universal formul\ae\
  for the classes of degeneracy loci of generic maps between flagged
  vector bundles.
  This concept was reinterpreted cohomologically in
  \cite{KM,Kazarian}, as giving the equivariant cohomology classes
  associated to matrix Schubert varieties; we recall this 
  interpretation in \S \ref{sec:matrix}.  
\end{enumerate}

In this paper we give an expeditious proof of the following known
results \cite{BB,KM}:

\begin{Theorem}\label{thm:main}
  For all $\pi\in S_\infty$, $A_\pi = C_\pi = G_\pi$.
\end{Theorem}

This will follow from a base case $w_0^n$ they share, where
$w_0^n(i) :=
\begin{cases}
  i &\text{if }i>n \\
  n+1-i  &\text{if }i\leq n ,
\end{cases}$

\begin{Lemma*}[The base case]\label{lem:base}
  For each of $P \in \{A,C,G\}$, 
  we have $P_{w_0^n} = \displaystyle\prod_{i,j \in [n],\ i+j\leq n} (x_i - y_j)$.
\end{Lemma*}

\noindent along with a recurrence they each enjoy:

\begin{Lemma*}[The \defn{co-transition formula}]\label{lem:cotransition}
  For each of $P \in \{A,C,G\}$, and $\pi \in S_n \setminus \{w_0^n\}$, 
  there exist $i$ such that $i + \pi(i) < n$. Pick $i$ minimum such. Then
  $$ (x_i - y_{\pi(i)}) \, P_\pi   = \sum \left\{ 
    P_\sigma\ :\ \sigma \in S_n, \ \sigma \gtrdot \pi, 
    \ \sigma(i) \neq \pi(i) \right\} $$
  where $\gtrdot$ indicates a cover in the Bruhat order.
\end{Lemma*}

The derivations of the co-transition formula in the three families are
to some extent parallel.
For $P=A$ we define the ``support'' of a polynomial and remove one point
from the support of $A_\pi$. In $P=C$ we (implicitly) study a subword
complex \cite{KM} whose facets correspond to pipe dreams for $\pi$, 
and delete a cone vertex from the complex. In $P=G$ we study a
hyperplane section of the matrix Schubert variety $\barX_\pi$,
which removes one $T$-fixed point from $\barX_\pi/T$.

In the remainder we recall the polynomials and prove the lemmata for
each of them. The word ``transition'' will appear in \S\ref{sec:schubert}, 
but the ``co-'' will only be explained in \S\ref{sec:trans}.
To further illuminate the connection between the pipe dream formula
and the co-transition formula, we include in \S\ref{sec:inductive}
an inductive formula that generalizes both and can be derived from either.

\subsection*{Acknowledgments.} It is a great pleasure
to get to thank Bill for so much mathematics, encouragement, and guidance
(especially in the practice and the importance of writing well; 
while my success has been limited I have at least always striven to emulate
his example). I thank Ezra Miller for his many key insights in \cite{KM}, 
even as I now obviate some of them here, and Bernd Sturmfels for his
early input to that project. I thank the referee for catching some
embarassing errors. Finally, this is my chance once more to
thank Nantel Bergeron, Sara Billey, Sergei Fomin, and Anatol Kirillov
for graciously accepting the terminology ``pipe dream''.
(See \cite{Bergeron}!)

\section{The double Schubert polynomials $(A_\pi)$}
\label{sec:schubert}

Define the \defn{nil Hecke algebra} $\ZZ[\partial]$ as having a $\ZZ$-basis
$\{ \partial_\pi\ :\ \pi \in S_\infty \}$ and the following very
simple product structure:
$$ \partial_\pi \partial_\rho := 
\begin{cases}
  \partial_{\pi \circ \rho} &\text{if }\ell(\pi \rho) = \ell(\pi) + \ell(\rho)\\
  0 &\text{otherwise, i.e. }\ell(\pi \rho) < \ell(\pi) + \ell(\rho).
\end{cases}
$$
Here
$\ell(\pi) := \#\left\{ (i,j)\in (\NN_+)^2\ :\ i<j,\ \pi(i)>\pi(j) \right\}$
denotes the number of inversions of $\pi$. So this algebra $\ZZ[\partial]$
is graded by $\deg \partial_\pi = \ell(\pi)$, 
and plainly is generated by the degree $1$ elements
$\{\partial_i := \partial_{r_i}\}$, modulo the \defn{nil Hecke relations}
$$ \partial_i^2 = 0\qquad\qquad
[\partial_i,\partial_j] = 0, \quad |i-j|>2 \qquad\qquad
\partial_i \partial_{i+1} \partial_i = \partial_{i+1} \partial_i \partial_{i+1}
$$

\newcommand\ul[1]{{\underline #1}}
This algebra has a module 
$\ZZ[\ul x, \ul y] := \ZZ[x_1,x_2,\ldots,y_1,y_2,\ldots]$ where the 
action is by \defn{divided difference operators} in the $x$ variables:
$$ \partial_i p := \frac{p - r_i(p)}{x_i - x_{i+1}} $$
Here $r_i \circlearrowright \ZZ[\ul x,\ul y]$
is the ring automorphism exchanging $x_i\leftrightarrow x_{i+1}$
and leaving all other variables alone. Since the numerator of $\partial_i p$
negates under switching $x_i$ and $x_{i+1}$, the long division algorithm for
polynomials shows that numerator to be a multiple of $x_i-x_{i+1}$,
so $\partial_i p$ is again a polynomial. To confirm that the above 
defines an action, one has to check the nil Hecke relations,
which is straightforward.

The action is linear in the $(y_i)$ variables, and the module comes with 
a $\ZZ[\ul y]$-linear augmentation $|_e:\, \ZZ[\ul x,\ul y] \to \ZZ[\ul y]$ 
setting each $x_i \mapsto y_i$. With this, we can define a pairing
\begin{eqnarray*}
  \ZZ[\partial] \times \ZZ[\ul x,\ul y] &\to& \ZZ[\ul y] \\
  (\partial_w, p) &\mapsto& (\partial_w(p))\,|_e
\end{eqnarray*}
Since the $\partial_w$ act $\ZZ[\ul y]$-linearly, it is safe
to extend the scalars in the nil Hecke algebra from $\ZZ$ to $\ZZ[\ul y]$,
and regard $\ZZ[\ul y]$ as our base ring for the two spaces being paired,
as well as the target of their pairing.

\begin{Proposition}
  This $\ZZ[\ul y]$-valued pairing of $\ZZ[\ul y][\partial]$ and
  $\ZZ[\ul x,\ul y]$ is perfect, so the basis
  $\{ \partial_\pi\ :\ \pi \in S_\infty \}$ has a dual $\ZZ[\ul y]$-basis
  $\{ A_\pi\ :\ \pi \in S_\infty \}$ of $\ZZ[\ul x,\ul y]$,
  called the \defn{double Schubert polynomials}.
  These are homogeneous with $\deg(A_\pi) = \ell(\pi)$.

  In this basis, the $\ZZ[\partial]$-module structure becomes
  $$ \partial_\pi A_\rho = 
  \begin{cases}
    A_{\rho \circ \pi^{-1}} & \text{if }
    \ell(\rho \circ \pi^{-1}) =     \ell(\rho) - \ell(\pi) \\
    0 &\text{otherwise, i.e. if }
    \ell(\rho \circ \pi^{-1}) >     \ell(\rho) - \ell(\pi) 
  \end{cases}
  $$
\end{Proposition}

There are enough fine references for Schubert polynomials
(e.g. \cite{YoungTableaux}) that we don't further recapitulate 
the basics here. Dual bases are always unique, and perfection of the
pairing is equivalent to existence of the dual basis. 
The usual proof of the existence starts with the base case $A_{w_0^n}$
as an {\em axiom,} defining the other double Schubert polynomials using
the module action stated in the proposition.

It remains to prove the co-transition formula (for $P=A$), which 
in the ``single'' situation (setting all $y_i \equiv 0$) is plainly
a Monk's rule calculation. Since the ``double'' Monk rule is not
a standard topic, and the references we found to it
(e.g. notes by D. Anderson from a course by Fulton)
use theory beyond the algebra definition above, we include a proof of
the co-transition formula appropriate to $P=A$.

One tool for studying double Schubert polynomials
is the $\ZZ[\ul y]$-algebra homomorphism 
$\ZZ[\ul x,\ul y] \to \ZZ[\ul y]$, $x_i \mapsto y_{\rho(i)} \forall i$,
called \defn{restriction to the point $\rho$}.
We'll write this as $f \mapsto f|_\rho$,
generalizing the case $\rho=e$ (the identity) we used above to define 
the pairing. Here is how it interacts with divided difference operators:
\begin{equation}
  \label{eqn:divdiff}
(\partial_i f)|_\rho
= \frac{f - r_i f}{x_i - x_{i+1}}\bigg|_\rho
= \frac{f|_\rho - (r_i f)|_\rho}{x_i|_\rho - x_{i+1}|_\rho}
= \frac{f|_\rho - f|_{\rho r_i}}{y_{\rho(i)} - y_{\rho(i+1)}}
\tag{$*$}  
\end{equation}

Define the \defn{support} $supp(f)$ 
of a polynomial $f \in \ZZ[\ul x,\ul y]$ 
by $supp(f) :=\{\sigma\in S_\infty\ :\ f|_\sigma \neq 0\}$.
It has a couple of obvious properties: $supp(pq) = supp(p)\cap supp(q)$,
and $supp(p+q) \subseteq supp(p) \union supp(q)$.

\begin{Proposition}\label{prop:pointrestriction}
  \begin{enumerate}
  \item $supp(\partial_i f) \subseteq supp(f) \ \union\ supp(f)\cdot r_i $
  \item If $supp(f) \subseteq \{\tau : \tau\geq \sigma\}$,
    then $supp(\partial_i f)
    \subseteq \{\tau : \tau\geq \min(\sigma,\sigma r_i) \}$.
  \item $supp(A_{w_0^n}) \cap S_n = \{w_0^n\}$
  \item $A_\pi|_\rho = 0$ unless $\rho\geq \pi$ in Bruhat order.
    (The converse holds, but we won't show it.)
  \item Let $\pi \in S_n$. 
    Then $A_\pi|_\pi \neq 0$. (A small converse to (4).)
  \item If $f|_\rho = 0$ for all $\rho \in S_\infty$, then $f=0$.
  \item There is an algorithm to expand a polynomial $p$ as a
    $\ZZ[\ul y]$-combination of double Schubert polynomials:
    look for a Bruhat-minimal element $\sigma$ of the support,
    subtract off $\frac{p|_\sigma}{A_\sigma|_\sigma} A_\sigma$ from $p$
    (recording the coefficient $\frac{p|_\sigma}{A_\sigma|_\sigma}$),
    and recurse until $p$ becomes $0$.
  \end{enumerate}
\end{Proposition}

\begin{proof}
  \begin{enumerate}
  \item 
    Use
    $(\partial_i f)|_\rho \ =\
    f|_\rho/(y_{\rho(i)} - y_{\rho(i+1)})\ -\ f|_{\rho r_i}/(y_{\rho(i)} - y_{\rho(i+1)})$
    from equation (\ref{eqn:divdiff}).
  \item This follows from (1) and the subword characterization of Bruhat order.
  \item This follows trivially from 
    $A_{w_0^n} = \prod_{i,j \in [n],\ i+j\leq n} (x_i - y_j)$.
  \item Fix $n$ such that $\pi,\rho \in S_n$,
    so $A_\pi = \partial_{\pi^{-1} w_0^n} A_{w_0^n}$.
    Let $Q$ be a reduced word for $\pi^{-1} w_0^n$.
    Apply (2); by the reducedness of $Q$ the $min(\sigma,\sigma r_i)$ is
    always $\sigma r_i$. By induction on $\# Q$ we learn
    $supp(\partial_{\pi^{-1} w_0^n} A_{w_0^n}) \subseteq \{\tau\ :\ \tau\geq\pi\}$,
    which is the result we seek.
  \item We use downward induction in weak Bruhat order
    from the easy base case $w_0^n$. If $\pi r_i \gtrdot \pi$, then
    $A_\pi|_\pi = (\partial_i A_{\pi r_i})|_\pi
    \propto (A_{\pi r_i}|_\pi - A_{\pi r_i}|_{\pi r_i}) 
    = - A_{\pi r_i}|_{\pi r_i} \neq 0$ using equation (\ref{eqn:divdiff})
    for the $\propto$, part (4) to kill the first term, and induction.
  \item Expand $f = \sum_{\pi\in S_\infty} c_\pi A_\pi$ in the $\ZZ[\ul y]$-basis 
    $\{A_\pi\}$ and, if $f\neq 0$, let $A_\rho$ be a summand appearing 
    (i.e. $c_\rho\neq 0$) with $\rho$ minimal in Bruhat order.
    Then $f|_\rho = \sum_{\pi} c_\pi A_\pi|_\rho = c_\rho A_\rho|_\rho$ by (4), 
    and this is $\neq 0$ by (5).
  \item In the finite $\ZZ[\ul y]$-expansion $p = \sum_\rho d_\rho A_\rho$, 
    if $\sigma$ is chosen minimal such that $d_\sigma \neq 0$, then
    $p|_\sigma = \sum_\rho d_\rho A_\rho|_\sigma = d_\sigma A_\sigma|_\sigma \neq 0$,
    so $\sigma$ lies in $p$'s support. Meanwhile, by (4) $\sigma$ must
    also be Bruhat-minimal in $p$'s support. When we perform the subtraction
    in the algorithm, the coefficient is $d_\sigma$, 
    and the number of terms in $p$ decreases.
  \end{enumerate}
\end{proof}

When we later learn $A=C=G$, then properties (4), (5) of the $(A_\pi)$ 
will also hold for $(C_\pi), (G_\pi)$, and we leave the reader
to seek direct proofs of them.

\begin{Proposition}[Equivariant Monk's rule]\label{prop:eqvtMonk}
  Let $\pi \in S_\infty$, $i>0$. Then
  $$ (x_i-y_{\pi(i)}) A_\pi = 
  \sum_{\rho \gtrdot \pi} A_\rho
  \begin{cases} +1 &\text{if }    \rho(i) > \pi(i) \\
    -1 &\text{if }    \rho(i) < \pi(i) \\
    0 &\text{if }    \rho(i) = \pi(i) 
  \end{cases}
  $$
\end{Proposition}

\begin{proof}
  Using the algorithm from proposition \ref{prop:pointrestriction}(7), 
  and also proposition \ref{prop:pointrestriction}(4), 
  we know that the expansion $f = \sum_\rho c_\rho A_\rho$ can only involve
  those $\rho \geq$ elements of $f$'s support.
  The support of $(x_i-y_{\pi(i)}) A_\pi$ lies in 
  $\{\rho \in S_\infty: \rho \geq \pi\} \setminus \{\pi\}
  = \{\rho \in S_\infty: \rho > \pi\}$.
  The only elements of that set with length $\leq \deg (x_i-y_{\pi(i)}) A_\pi$
  are $\{\rho \in S_\infty: \rho \gtrdot \pi\}$.  
  Hence the left-hand side, expanded in double Schubert polynomials, must 
  have constant coefficients, not higher-degree polynomials in $\ZZ[\ul y]$.
  (This is the sense in which the ``right'' extension of Monk's
  nonequivariant rule concerns multiplication by $x_i-y_{\pi(i)}$ not just $x_i$.
  There is of course another, equally ``right'', extension,
  computing $A_{r_i} A_\pi$.)

  If a polynomial $f$ is in the common kernel of the $\partial_j$
  operators, it must be symmetric in {\em all} the $\ul x$
  variables... which means $f$ must involve no $\ul x$ variables at
  all, i.e. $f\in \ZZ[\ul y]$.  If we also insist that $f|_e = 0$ then
  we may infer $f=0$.

  Both sides of our desired equation are homogeneous polynomials of
  the same degree, $\ell(\pi) + 1$, and with $f|_e = 0$.  
  By the argument above it suffices to show that
  $\partial_j$ LHS $= \partial_j$ RHS for all $j$.
  There are five cases: $(j=i$ or $j=i-1) \times (\pi(i)>\pi(i+1)$ or
  $\pi(i)<\pi(i+1))$, and $j \neq i,i-1$, each of
  which one can check using the (itself easily checked) twisted Leibniz rule
  $\partial_i (fg) = (\partial_i f)g + (r_i f)(\partial_i g)$
  along with induction on $\ell(\pi)$. We explicitly check the 
  most unpleasant of the five cases:  $j=i, \pi(i) > \pi(i+1)$. 
  \begin{eqnarray*}
    \partial_i (x_i-y_{\pi(i)}) A_\pi 
    &=& A_\pi + (x_{i+1} - y_{\pi(i)}) \partial_i A_\pi    \\
    &=& A_\pi + (x_{i+1} - y_{(\pi r_i)(i+1)}) A_{\pi r_i} 
        \qquad\text{now use induction} \\
    &=& A_\pi + \sum_{\sigma \gtrdot \pi r_i} A_\sigma
    \begin{cases} +1 &\text{if }    \sigma(i+1) > (\pi r_i)(i+1) = \pi(i) \\
      -1 &\text{if } \sigma(i+1) < (\pi r_i)(i+1) = \pi(i) \\
      0 &\text{if } \sigma(i+1) = (\pi r_i)(i+1) = \pi(i) 
    \end{cases} \\
    \sum_{\rho \gtrdot \pi} \partial_i A_\rho
    \begin{cases} +1 &\text{if }    \rho(i) > \pi(i) \\
      -1 &\text{if }    \rho(i) < \pi(i) \\
      0 &\text{if }    \rho(i) = \pi(i) 
    \end{cases}
      &=&     \sum_{\rho \gtrdot \pi,\, \rho(i)>\rho(i+1)} A_{\rho r_i}
    \begin{cases} +1 &\text{if }    \rho(i) = (\rho r_i)(i+1)> \pi(i) \\
      -1 &\text{if }    \rho(i) = (\rho r_i)(i+1)< \pi(i) \\
      0 &\text{if }    \rho(i) = (\rho r_i)(i+1)= \pi(i) 
    \end{cases}
  \end{eqnarray*}
  Each term $\sigma$ in the first corresponds to the
  $\rho = \sigma r_i$ term in the second.
\end{proof}

\begin{proof}[Proof of the co-transition formula for $P=A$.]
  \junk{
  Since $((x_i - y_{\pi(i)})\, A_\pi)|_\rho = 
  (x_i - y_{\pi(i)})|_\rho  A_\pi|_\rho = (y_{\rho(i)} - y_{\pi(i)}) A_\pi|_\rho$,
  the support of $(x_i - y_{\pi(i)})\, A_\pi$ is contained in 
  $\{ \rho\ :\ \rho \geq \pi, \rho(i) \neq \pi(i) \}
  \subset \{ \rho\ :\ \rho > \pi \}$.

  Such $\rho$ have $\ell(\rho) \geq 1 + \ell(\pi)$, but
  $\deg(x_i - y_{\pi(i)})\, A_\pi) = 1 + \deg(A_\pi) = 1 + \ell(\pi)$, 
  so the only possible terms in the expansion (from the algorithm)
  have $\rho \gtrdot \pi$. Since all these $\rho$ are the same length,
  they are incomparable in Bruhat order, so the expansion algorithm can 
  come to them in any order. The coefficient of $A_\rho$, 
  computed as $(y_{\rho(i)} - y_{\pi(i)})\, A_\pi|_\rho \big / A_\rho|_\rho$ 
  in the expansion algorithm, is $1$ by lemma \ref{lem:gtrdot}.
  }
  We need to check that each $\rho$ term in the 
  equivariant Monk rule has $\rho(i) \in (\pi(i), n]$,
  so as to only get positive terms and only from $\rho \in S_n$.

  Since $\pi$ has only descents before $i$ (by choice of $i$), 
  we know $\rho = \pi \circ (i\leftrightarrow b)$ with $i<b$,
  i.e. $\rho(i) = \pi(b) > \pi(i)$. 

  By choice of $i$, 
  we have $\pi = n\, n$-$1\ldots n$-$i$+$2\, \pi(i)\ldots \pi(n)$
  with $\pi(i) < n-i+1$. Hence $\exists j \in (i,n]$ with
  $\pi(j) = n-i+1 \in (\pi(i),n+1)$. The covering relations in $S_\infty$
  don't allow us to switch positions $i,n+k$ if some position
  $j \in (i,n+k)$ has $\pi(j) \in (\pi(i),\pi(n+k)=n+k)$.
\end{proof}

Lascoux' {\em transition formula} \cite{LascouxTransition}
for double Schubert polynomials is also based on Monk's rule,
but doesn't include implicit division like the co-transition formula does.
(It is worth noting that {\em each} of the summands on the right-hand side
of the co-transition formula is divisible by $x_i-y_{\pi(i)}$, not merely
their total.) We discuss the connection in \S \ref{sec:trans}.

\section{The pipe dream polynomials $(C_\pi)$}
\label{sec:pipedream}

Index the squares in the Southeastern quadrant of the plane 
using matrix coordinates $\{ (a,b)\ :\ a,b \in \{1,2,3,\ldots\} \}$.
A \defn{pipe dream} is a filling of that quadrant 
with two kinds of tiles, mostly $\textelbow$ but finitely many 
$\textcross$, such that no two pipes cross twice%
\footnote{In subtler contexts than considered here, one {\em does} allow pipes
  to cross twice, and instead refers to the pipe dreams without
  double crossings as ``reduced pipe dreams''. See \S\ref{sec:K}.}
We label the pipes $1,2,3,\ldots$ across the top side, 
and speak of ``the $1$-pipe of $D$'', ``the $2$-pipe of $D$'', and so on.
For example, the left two diagrams below are pipe dreams, the right one not:

$$
\begin{array}{ccccc}
       &\perm1{}&\perm2{}&\perm3{}&\\
\petit1&   \jr  &   \+  &   \jr & \ \cdots    \\
\petit3&   \jr  &   \jr  &   \jr \\
\petit2&   \jr  &   \jr  &   \jr \\
       &\vdots  &        &       & \ddots
\end{array}
\qquad\qquad\hbox{}\qquad\
\begin{array}{ccccc}
       &\perm1{}&\perm2{}&\perm3{}&\\
\petit1&   \jr  &   \jr  &   \jr & \ \cdots    \\
\petit3&   \+   &   \jr  &   \jr \\
\petit2&   \jr  &   \jr  &   \jr \\
       &\vdots  &        &       & \ddots
\end{array}
\qquad\qquad\hbox{}\qquad\
\begin{array}{ccccc}
       &\perm1{}&\perm2{}&\perm3{}&\\
&   \jr  &   \+   &   \jr & \ \cdots    \\
&   \+   &   \jr  &   \jr \\
&   \jr  &   \jr  &   \jr \\
       &\vdots  &        &       & \ddots
\end{array}
$$

Because of the no-double-crossing rule, if we regard a pipe dream $D$
for $\pi$ as a wiring diagram for $\pi$, 
it's easy to see that the number of $\textcross$ is exactly $\ell(\pi)$.

To a pipe dream $D$ we can associate a permutation $\pi$
by reading off the pipe labels down the left side,
and say that $D$ is a ``pipe dream for $\pi$''.
With this we can define the \defn{pipe dream polynomials}:
$$ C_\pi := \sum_{\text{pipe dreams $D$ for $\pi$}} \ \ 
\prod_{\text{crosses $+$ in $D$}} 
(x_{\text{row}(+)} - y_{\text{col}(+)}) $$

$$
\begin{matrix}
\text{Example: }C_{1423} 
   =& (x_2-y_1)(x_2-y_2) &+& (x_2-y_1)(x_1-y_3) &+& (x_1-y_2)(x_1-y_3) \\
  & \begin{array}{ccccc}
       &\perm1{}&\perm2{}&\perm3{}&\perm4{}\\
\petit1&   \jr  &   \jr  &   \jr  &  \je   \\
\petit4&   \+   &   \+   &   \je  &\\
\petit2&   \jr  &   \je  &        &\\
\petit3&   \je  &        &        &\\
\end{array}
 && \begin{array}{ccccc}
       &\perm1{}&\perm2{}&\perm3{}&\perm4{}\\
\petit1&   \jr  &   \jr  &   \+   &  \je   \\
\petit4&   \+   &   \jr  &   \je  &\\
\petit2&   \jr  &   \je  &        &\\
\petit3&   \je  &        &        &\\
\end{array}
 && \begin{array}{ccccc}
       &\perm1{}&\perm2{}&\perm3{}&\perm4{}\\
\petit1&   \jr  &   \+   &   \+   &  \je   \\
\petit4&   \jr  &   \jr  &   \je  &\\
\petit2&   \jr  &   \je  &        &\\
\petit3&   \je  &        &        &\\
\end{array}
\end{matrix}
$$
where we skip drawing any of the pipes outside the triangle 
$\{(a,b)\ :\ a+b\leq n\}$, 
as will be justified by lemma \ref{lem:wiggler} below.

The main idea of the proof of the co-transition formula for 
the $\{C_w\}$ polynomials is easy to explain.
Let $\calD_1$ be the set of pipe dreams for $w$,
and 
$$ \calD_2 := \Union\ \{\text{the pipe dreams for }w'\ :\ w' 
\text{ occurs in the co-transition formula}\}. $$
Our goal (which will take some doing) 
is to show that the maps
$$
\begin{array}{rclcrclr}
  \qquad \qquad \calD_1 &\to& \calD_2 &\qquad&    \calD_2 &\to& \calD_1& \\
  D\textelbow &\mapsto& D\textcross &&
    D\textcross &\mapsto& D\textelbow 
        &
\end{array}
$$
that place, or remove, a $\textcross$ at position $(i,\pi(i))$
have the claimed targets $\calD_2$, $\calD_1$. 
The maps are then obviously inverse, and the co-transition formula 
will follow easily. 

\begin{Lemma}\label{lem:wiggler}
  Let $w\in \Sinf$, and $i\in \{1,2,\ldots\}$ such that
  $\forall j\in \{1,2,\ldots\},$ $i<j \iff w(i)<w(j)$. 
  Then the pipe that enters from the North in column $i$ only goes
  through $\textelbow$ tiles, no $\textcross$, coming out at row $w(i)$.
  Consequently, if $w\in \Sn$ and $D$ is a pipe dream for $w$,
  then there are no $\textcross$ tiles outside the triangle
  $\{ (a,b)\ :\ a+b \leq n \}$.
\end{Lemma}

\begin{proof}
  If $i<j$ and $w(i)<w(j)$, then the $i$-pipe starts and ends
  Northwest of the $j$-pipe. By the Jordan curve theorem these
  two pipes cross an even number of times, and since $D$ is a pipe dream, 
  that even number is $0$. The opposite argument (Southeast) works
  if $i>j$ and $w(i)>w(j)$.
  Doing this for all $j\neq i$, we find that the $i$-pipe crosses no other pipe,
  i.e. it goes only through $\textelbow$ tiles, ruling out $\textcross$ tiles
  in the adjacent diagonals $\{(a,b)\ :\ a+b = i-1,i\}$. 
  Finally, if $w\in \Sn$ then each $i>n$ satisfies the condition.
\end{proof}

\begin{proof}[Proof of the base case for $P=C$.]
  The number of squares in the triangle 
  $\{ (a,b)\ :\ a+b \leq n \}$ is $n\choose 2$, which is also $\ell(w_0^n)$.
  As such, every one of them must have a $\textcross$ in a pipe dream
  for $w_0^n$, making the pipe dream for $w_0^n$ unique.
  Then the definition of $C_\pi$ gives the base case.
\end{proof}

\begin{Lemma}\label{lem:crosses}
  Let $\pi,i,\rho$ be as in the co-transition formula. If $D$ is a
  pipe dream for $\pi$, then the leftmost $\textelbow$ in rows $1,2,\ldots,i$
  of $D$ occurs in column $\pi(1),\pi(2),\ldots,\pi(i)$ respectively.
  If $D'$ is a pipe dream for $\rho$,
  then the same is true in rows $1,2,\ldots,i-1$ but in 
  row $i$ the leftmost $\textelbow$ occurs strictly to the right
  of column $\pi(i)$.
\end{Lemma}

\begin{proof}
  Assume that the claim is established for each row above the $j$th.
  Start on the North side of $D$ in column $\pi(j)$, and follow that pipe down.
  By our inductive knowledge of rows $1\ldots j-1$, 
  and the fact that $\pi(1) > \pi(2) > \cdots > \pi(i)$ by choice of $i$,
  this pipe will go straight down through $j-1$ crosses to the $j$th row.
  Since it then needs to exit on the $j$th row, it must turn West
  in matrix position $(j,\pi(j))$, and go due West through 
  only $\textcross$ in columns $1,\ldots,\pi(j)-1$ of that row.

  Exactly the same analysis holds for $\rho$, except that $\rho(i)>\pi(i)$.
\end{proof}

The following technical lemma is key. 

\begin{Lemma}\label{lem:holey}
  Let $D$ be a filling of the Southeastern quadrant with finitely
  many $\textcross$, the rest $\textelbow$, except {\em with an empty square}
  at $(a,b)$. Let $N$, $S$, $E$, $W$ denote 
  the four pipes coming out of $(a,b)$ 
  in those respective compass directions and call the
  remaining pipes the ``old pipes''.
  Let $D\textelbow$, $D\textcross$ denote $D$ with the respective tile
  inserted at $(a,b)$; these have ``new pipes''
  $WN,ES$ in $D\textelbow$ and $NS,EW$ in $D\textcross$. Assume that:
  \begin{enumerate}
  \item Every square West of $E$ except the hole $(a,b)$ (so,
    in rows $1,\ldots,a$) has a $\textcross$,
    so in particular, the pipes $N$ and $W$ are straight.
  \end{enumerate}
  Then if $D\textelbow$, is a pipe dream, so is $D\textcross$.
  If in addition we assume
  \begin{enumerate}
  \item[(2)] No old pipe has North end between $N$ and $E$'s North ends 
    while also having West end between $W$ and $S$'s West ends
  \end{enumerate}
  then $D\textcross$ being a pipe dream implies $D\textelbow$ is a pipe dream.
\end{Lemma}

We give an example to refer to while following the case analysis
in the proof.
$$
D\textelbow = 
\begin{array}{cccccccc}
  && &\perm {WN}{} && \perm {ES}{} \\
  &&\perm1{}&\perm2{}&\perm3{}&\perm4{}&\perm5{}&\perm6{}\\
  &\petit6&   \+   &   \+   &   \+   &  \+   &  \+   &  \je  \\
\petit {WN}  &\petit2&   \+   & \jr   &   \+   &  \jr   &  \je   \\
 &\petit 1&   \jr   & \jr &   \jr   &  \je   \\
  \petit {ES}  &\petit 4 &   \jr  &   \jr   &   \je  &\\
 &\petit5&   \+  &   \je  &        &\\
 &\petit3&   \je  &        &        &\\
\end{array} \
D = 
\begin{array}{cccccccc}
  && &\perm N{} && \perm E{} \\
  &&\perm1{}&\perm2{}&\perm3{}&\perm4{}&\perm5{}&\perm6{}\\
  &\petit6&   \+   &   \+   &   \+   &  \+   &  \+   &  \je  \\
  \petit W&&   \+   &\rlap{\!$\square$}   &   \+   &  \jr   &  \je   \\
 &\petit 1&   \jr   & \jr &   \jr   &  \je   \\
  \petit S  &&   \jr  &   \jr   &   \je  &\\
 &\petit5&   \+  &   \je  &        &\\
 &\petit3&   \je  &        &        &\\
\end{array} \
D\textcross = 
\begin{array}{cccccccc}
  && &\perm {NS}{} && \perm {EW}{} \\
  &&\perm1{}&\perm2{}&\perm3{}&\perm4{}&\perm5{}&\perm6{}\\
  &\petit6&   \+   &   \+   &   \+   &  \+   &  \+   &  \je  \\
\petit{EW}  &\petit 4&   \+   &\+   &   \+   &  \jr   &  \je   \\
 &\petit 1&   \jr   & \jr &   \jr   &  \je   \\
  \petit {NS}  &\petit 2  &   \jr  &   \jr   &   \je  &\\
 &\petit5&   \+  &   \je  &        &\\
 &\petit3&   \je  &        &        &\\
\end{array}
$$

\begin{proof}
  Say $D\textelbow$ is a pipe dream, i.e.
  its new pipes $WN$ and $ES$ don't cross any other pipe twice;
  in particular no old pipe crosses any of $N,S,E,W$ twice.
  We need to make sure that in $D\textcross$
  the two new pipes $NS$ and $EW$ don't cross any old pipe twice.
  Equivalently, no old pipe should cross both $W$ and $N$,
  or both $E$ and $S$.
  Exactly the same analysis will hold for the opposite direction:
  if $D\textcross$ at $(a,b)$ is a pipe dream,
  we need show that no old pipe crosses both $E$ and $S$,
  or both $N$ and $W$.

  If a pipe (in either $D\textelbow$ or $D\textcross$) crosses $N$
  going West, then by condition (1) it goes straight West from there
  and cannot cross $W$ or $S$.  Similarly, if a pipe crosses $W$ going
  North, then by condition (1) it goes straight North from there and
  cannot cross $N$ or $E$.

  That rules out double-crossing $NS$, $EW$, and $WN$,
  so is already enough to establish our first conclusion
  ($D\textelbow$ a pipe dream $\implies D\textcross$ a pipe dream).
  What remains for the second conclusion is to show that, if $D\textcross$ 
  is a pipe dream, then no old pipe should cross both $E$ and $S$.

  Let $i,j$ denote the respective columns of the tops of $N$, $E$.
  If $h<i$, then the $h$-pipe stays West of $E$. If $h>j$, and
  the $h$-pipe crosses $E$, then it does so horizontally, at which point
  it continues due West and stays above $S$. 
  Finally, if $i < h < j$, then by condition (2) the $h$-pipe
  has West end either above $W$'s West end or below $S$'s West end.
  In the first case, the $h$-pipe stays above $W$ hence above $S$.
  In the latter case, the $h$-pipe begins and ends Southeast of
  the $NS$ pipe in $D\textcross$, so doesn't cross it at all,
  hence doesn't cross $S$. 
\end{proof}

\begin{proof}[Proof of the co-transition formula for $P=C$]
  Let $\calD_1$ be the set of pipe dreams for $w$,
  and 
$$ \calD_2 := \Union\ \{\text{the pipe dreams for }w'\ :\ w' 
  \text{ occurs in the co-transition formula}\}. $$
  Let $(a,b) = (i,w(i))$. Our goal is to show that the maps
  $$
  \begin{array}{rclcrclr}
    \qquad \qquad \calD_1 &\to& \calD_2 &\qquad&    \calD_2 &\to& \calD_1& \\
    D\textelbow &\mapsto& D\textcross &&
    D\textcross &\mapsto& D\textelbow 
        &\qquad\text{as in lemma \ref{lem:crosses}}
  \end{array}
  $$
  have the claimed targets $\calD_2$, $\calD_1$. 

  Let $D\textelbow \in \calD_1$. By our choice of $i$ from the
  co-transition formula, and of $(a,b) = (i,w(i))$, lemma \ref{lem:crosses} 
  establishes condition (1) of lemma \ref{lem:holey}. 
  Hence $D\textcross$ is a pipe dream for some $w' = w(i\leftrightarrow j)$. 
  Since $D\textcross$ has one more crossing than $D\textelbow$, 
  we infer $\ell(w') = \ell(w)+1$, 
  so $w' \gtrdot w$. Consequently $D\textcross \in \calD_2$.

  Now start from $D\textcross \in \calD_2$,
  a pipe dream for some $w'$; we want to show that $D\textelbow \in \calD_1$. 
  Again our choices of $i$ and $(a,b)$ establish condition (1)
  of lemma \ref{lem:holey}.
  Define $j$ so that the $EW$ pipe of $D\textcross$ is the $j$-pipe,
  i.e. $E$ exits the North side in column $j$.
  Since $w' \gtrdot w$, we verify condition (2) of lemma \ref{lem:holey}.
  Hence $D\textelbow \in \calD_1$.

  Each $\textcross$ inserted at $(i,\pi(i))$ contributes a factor
  of $x_i-y_{\pi(i)}$ in the formula for $C$-polynomials, so while 
  the bijection above corresponds pipe dreams for $C_\pi$ to
  those for $\{C_\rho\}$, the induced equality of polynomials is
  between $\sum_\rho C_\rho$ and $(x_i-y_{\pi(i)})\, C_\pi$,
  giving the co-transition formula.
\end{proof}

\section{The matrix Schubert classes $(G_\pi)$}
\label{sec:matrix}

Define a \defn{matrix Schubert variety} $\barX_\pi \subseteq M_n(\CC)$,
for $\pi\in S_n$ or more generally\footnote{Indeed, once one allows
  partial permutation matrices there's no need for the matrices to
  be square, but square will suffice for our application.}
a {\em partial} permutation matrix, by
$$ \barX_\pi := \overline{B_- \pi B_+} \qquad\text{ closure taken in }M_n(\CC) $$
where $B_-,B_+$ are respectively 
the groups of lower and upper triangular matrices
intersecting in the diagonal matrices $T$.
The equations defining $\barX_\pi$ were determined in \cite[\S 3]{Fulton92}.

Define the \defn{matrix Schubert class}
$$
\begin{array}{rcll}
  G_\pi := \big[\,\barX_\pi\big] 
&\in& H^*_{B_- \times B_+}(M_n(\CC)) &\quad\text{using the smoothness of }M_n(\CC)\\
&\iso& H^*_{B_- \times B_+}(pt) &\quad \text{using the contractibility of }M_n(\CC)\\
&\iso& H^*_{T \times T}(pt) &\quad \text{since $B_-,B_+$ retract to $T$} \\
&\iso& \ZZ[x_1,\ldots,x_n, y_1,\ldots, y_n] 
 &\quad \text{using the usual isomorphism $T\iso (\CC^\times)^n$}
\end{array}
$$
in equivariant cohomology. 

Though we won't use it, we recall the connection to degeneracy loci
\cite{Kazarian}.
If we follow the Borel definition of $(B_- \times B_+)$-equivariant cohomology, 
based on the ``mixing space'' construction 
$Z(N) := N \times^{B_-\times B_+} E(B_-\times B_+))$,
the maps $\barX_\pi \into M_n(\CC) \to pt$ give a triangle
$$
\begin{matrix}
  Z(\barX_\pi) & \into & Z(M_n(\CC)) \\
  & \searrow & \downarrow \\
  & & Z(pt) & = & B(B_- \times B_+)
\end{matrix}
$$
where $B(B_- \times B_+)$ is the classifying space for
principal $(B_- \times B_+)$-bundles.

With this, $[Z(\barX_\pi)]$ defines a class in $H^*(Z(M_n(\CC)))$.
Since $\downarrow$ is a vector bundle hence a homotopy equivalence,
we can also take $[Z(\barX_\pi)]$ 
as a class in $H^*(B(B_-\times B_+)) =: H^*_{B_-\times B_+}$.

\newcommand\otni\hookleftarrow
Now consider a space $N$ bearing a flagged vector bundle 
$V_1 = V_1^{(n)} \otni V_1^{(n-1)} \otni \ldots\otni V_1^{(1)}$ and
a co-flagged vector bundle 
$V_2 = (V_2)_{(n)} \onto (V_2)_{(n-1)} \onto \ldots (V_2)_{(1)}$ 
(of course, in finite dimensions flagged
and co-flagged are the same concept), plus a generic map $\sigma: V_1 \to V_2$.
Since such pairs of bundles are classified by maps into 
$B(B_- \times B_+)$, we can enlarge the diagram to
$$
\begin{matrix}
  Z(\barX_\pi) & \into & Z(M_n(\CC)) \\
  & \searrow & \downarrow &\nwarrow\sigma\\
  & & B(B_- \times B_+) &\from\phantom{\sigma} &N
\end{matrix}
$$
and the genericity of $\sigma$ becomes its transversality to $Z(\barX_\pi)$.
Consequently, and using the equations from \cite[\S 3]{Fulton92}
defining $\barX_\pi$,
$$ \sigma^*\left(\left[\,\barX_\pi\right]\right)
= \left[
  \left\{ x \in N\ :\ \forall i,j, 
    \begin{array}{rl}
       \rank\left(
        V_1^{(i)} \into V_1 \xrightarrow\sigma V_2 \onto (V_2)_{(j)}
      \text{ over the point }x
      \right) \\
      \leq \rank(\text{NW $i\times j$ submatrix of $\pi$} )     
    \end{array}
    \right\}
    \right]
$$
i.e. $G_\pi = \left[\, \barX_\pi\right] \in \ZZ[x_1,\ldots,x_n,y_1,\ldots,y_n]$
is providing a universal formula for the class of the $\pi$ degeneracy
locus of the generic map $\sigma$. The principal insight is the dual
role of the space $B(B_-\times B_+)$, as the classifying space for
pairs of bundles and also as the base space of equivariant cohomology.

\begin{Lemma}\label{lem:indepn}
  The definition above is independent of $n$, so long as $\pi \in S_n$.
\end{Lemma}

\begin{proof}
  The equations defining $\barX_\pi$, determined in \cite[\S 3]{Fulton92},
  depend only on the matrix entries northwest of the Fulton essential set
  of $\pi$, which is independent of $n$. Hence enlarging $n$ to $m$ 
  amounts to crossing both $M_n(\CC)$, and $\barX_\pi$, by the same
  irrelevant vector space consisting of matrix entries 
  $\{ (i,j)\ :\ i>n\text{ or }j>n\}$.
\end{proof}

\begin{proof}[Proof of the base case for $P=G$]
  The Rothe diagram of $w_0^n$ is the triangle
  $\{ (a,b)\ :\ a+b \leq n \}$, so by \cite[\S 3]{Fulton92}
  the equations defining $\barX_{w_0^n}$ are that each entry $m_{ab}$
  in that triangle must vanish. This $\barX_{w_0^n}$ thus
  being a complete intersection, its class is the product of the 
  $(T\times T)$-weights $x_a-y_b$ of its defining equations $m_{ab}=0$,
  giving the base case formula.
\end{proof}

The following geometric interpretation of the Rothe diagram 
seems underappreciated:

\begin{Lemma}\label{lem:rothe}
  The tangent space $T_\pi \barX_\pi$ is $(T\times T)$-invariant
  (even though the point $\pi$ itself is not!), spanned by the matrix entries
  {\em not} in the Rothe diagram of $\pi$. In particular
  $\deg G_\pi = \#($the Rothe diagram$)$, which is in turn
  $\min \{ \ell(\rho): \rho$ 
  a permutation matrix with $\pi$ as its NW corner$\}$.
\end{Lemma}

\newcommand\lie[1]{{\mathfrak #1}}

\begin{proof}
  The tangent space to a group orbit is the image of the Lie algebra,
  $\lie{b}_- \pi + \pi \,\lie{b}_+$. The diagonal matrices (from either side)
  scale the nonzero entries of $\pi$, and the $\lie{n}_-$, $\lie{n}_+$ copy
  those entries to the South and East, recovering the usual
  death-ray definition of the complement of the Rothe diagram.

  For the ``in turn'' claim, observe that if $\rank(\pi) = n-k$ then there
  is a unique way to extend $\pi$ to $\rho \in S_{n+k}$ without 
  adding any boxes to the Rothe diagram, and $\ell(\rho)$ is the size of that
  diagram (of $\rho$ or of $\pi$).
\end{proof}

That lemma also gives a nice proof of proposition \ref{prop:pointrestriction}(5)
for $G_\pi$, though we won't need an independent one.

To compute other tangent spaces of $\barX_\pi$,
soon, we prepare a technical lemma.

\begin{Lemma}\label{lem:entry}
  Let $\rho \geq \pi \in S_n$. For $i,j\leq n$ 
  denote the NW $i\times j$ rectangle of $M$ by $M_{[i][j]}$.
  Let $a',b'$ be such that $\rank \pi_{[a'][b']} = \rank \rho_{[a'][b']}$.
  Let $(a,b)$ be in that rectangle, such that the $a$ row and $b$ column
  of $\rho_{[a'][b']}$ vanish. Then the $(a,b)$ entry vanishes on
  every element of $T_\rho \barX_\pi$.
\end{Lemma}

\begin{proof}
  Let $R$ be the nonzero rows and $C$ the nonzero columns of
  $\rho_{[a'][b']}$ (so $\#R = \#C = \rank \rho_{[a'][b']} = \rank \pi_{[a'][b']}$,
  and $a\notin R,b\notin C$ by the assumption on $(a',b')$).
  Consider the determinant that uses rows $R \cup \{a\}$ 
  and columns $S \cup \{b\}$; it is one of Fulton's required equations
  for $\barX_\pi$. We apply it to
  the infinitesimal perturbation $\rho + \varepsilon Z$.
  By construction this is $\varepsilon Z_{ab} + O(\varepsilon^2)$, 
  so for $Z$ to be in $T_\rho \barX_\pi$ we must have $Z_{ab} = 0$.
\end{proof}

This allows for a second proof of lemma \ref{lem:rothe}, when $\rho = \pi$; 
we can take $(a',b')=(a,b)$ for each $(a,b)$ in the Rothe diagram. 
These equations are already enough to cut down $\dim T_\pi \barX_\pi$ 
to the right dimension, and the tangent space can't get any 
lower-dimensional than that, so we have successfully
determined it from these determinants.
Having two proofs shows that the equations from \cite[\S 3]{Fulton92}
define a {\em generically} reduced scheme supported on $\barX_\pi$, 
unlike Fulton's stronger result that they actually define $\barX_\pi$.

\begin{Lemma}\label{lem:Tcover}
  Let $\rho = \pi \circ (i\leftrightarrow j) \gtrdot \pi$, 
  with $i<j\leq n$. 
  Then $T_\rho \barX_\pi \cap \{M : m_{i,\pi(i)} = 0\} = T_\rho \barX_\rho$.
\end{Lemma}

\begin{proof}
  The diagrams of $\pi$ and $\rho$'s agree except on the boundary of
  the ``flipping rectangle'' with NW corner $(i,\pi(i))$ and SE 
  corner $(j,\pi(j))$. Let $(a,b)$ in $\rho$'s diagram;
  we need to find an $(a',b')$ to apply lemma \ref{lem:entry} to.

  For $(a,b)$ outside the flipping rectangle, hence also in $\pi$'s
  diagram, we can use $(a',b') = (a,b)$ as explained directly
  after lemma \ref{lem:entry}. In other cases we will need to move
  Southeast from $(a,b)$ to $(a',b')$, without hitting the entries
  $(a,\pi(a))$ or $(\pi^{-1}(b),b)$ making lemma \ref{lem:entry} inapplicable.

  For $(a,b)$ in the interior of the flipping rectangle, we have $i < a < j$.
  Since $(a,b)$ is in $\pi$'s diagram, $a < \pi^{-1}(b)$. 
  We know that $(\pi^{-1}(b),b)$ isn't in the flipping rectangle 
  since $\rho \gtrdot \pi$, so $\pi^{-1}(b) < i$ or $\pi^{-1}(b) > j$.
  That first case is impossible since we'd have $a < \pi^{-1}(b) < i < a$,
  so we know $\pi^{-1}(b) > j$. This means we can safely go below $(a,b)$
  to $(a',b'):= (a,j)$, 
  with the benefit that $\rank \pi_{[a][j]} = \rank \rho_{[a][j]}$
  and we can apply lemma \ref{lem:entry}.
  
  It remains to handle the boundary of the flipping rectangle.
  The South edge $(j,*)$ and East edge $(*,\pi(i)$) are not in $\rho$'s diagram,
  so not at issue. Across the top edge $a=i$ and $\pi(i) < b < \pi(j)$,
  if $(i,b)$ is in $\rho$'s diagram then $\rho^{-1}(b) > i$, and 
  similarly to the above, we learn $\rho^{-1}(b) > j$. So once again
  we can safely go below $(a,b)$ to $(a',b'):= (a,j)$, 
  with the benefit that $\rank \pi_{[a][j]} = \rank \rho_{[a][j]}$
  and we can apply lemma \ref{lem:entry}.

  Finally, $(i,\pi(i))$ is in $\rho$'s diagram, but is killed by
  the intersection with $\{M : m_{i,\pi(i)} = 0\} = T_\rho \barX_\rho$
  rather than by a determinantal condition.

  This defines a vector space of dimension $\dim\barX_\pi - 1 = \dim\barX_\rho$,
  and $\dim T_\rho (\barX_\pi \cap \{M : m_{i,\pi(i)} = 0\})$ has at
  least that dimension, so we have found it.
\end{proof}

It will actually be more convenient to prove a slightly more general
formula than the co-transition formula as stated in \S\ref{sec:overview}.
Define the \defn{dominant part of} $\pi$'s Rothe diagram to be the 
boxes connected to the NW corner (this may be empty, when $\pi(1)=1$).
These are exactly the matrix entries $(a,b)$ such that $m_{ab} \equiv 0$
on $\barX_\pi$. (A permutation is ``dominant'' in the usual sense
if the dominant part
is the entire Rothe diagram, hence the terminology.) The $(i,\pi(i))$ of the
co-transition formula was picked to be
\begin{itemize}
\item just outside of the dominant part of $\pi$'s diagram
\item while still in the NW triangle,
\end{itemize}
and to be the Northernmost such ($i$ least such).
However, the co-transition formula holds {\em for any $(i,\pi(i))$
  satisfying the two bulleted conditions.}
This generalization would have made the proof in \S\ref{sec:pipedream}
more complicated, but of course once we know $C = G$ then we know 
that the $(C_\pi)$ also satisfy this more general formula.
Notice that this formula is stable under incrementing $n$ while not
changing the Rothe diagram (e.g. replacing $\pi \in S_n$ 
by $\pi \oplus I_1 \in S_{n+1}$, or a more complicated possibility 
if $\pi$ is a partial permutation).

\begin{Lemma}\label{lem:stable}
  Let $\pi \in S_n \setminus \{w_0^n\}$ 
  stabilize to $\rho \in S_{n+1}$, so $\rho(n+1) = n+1$
  and $\pi,\rho$ have the same Rothe diagram. Pick $(i,\pi(i))$ just outside
  the dominant part of this diagram, such that $i+\pi(i) \leq n$.
  Then this more general co-transition formula is the same for
  $\pi,\rho$; there aren't extra terms in $S_{n+1}$ for the formula for
  $(x_i - y_{\rho(i)})P_\rho$.
\end{Lemma}

(Of course this independence {\em follows} from the co-transition formula
and the linear independence of the $G$ polynomials, neither of which
we've proven yet.)

\begin{proof}
  Let $\sigma \gtrdot \rho$, so $\sigma = \rho \circ (a \leftrightarrow b)$
  with $a<b$, $\rho(a)<\rho(b)$, 
  and $c \in (a,b) \implies \rho(c) \notin (\rho(a),\rho(b))$
  (``no position $c$ is in the way when swapping positions $a,b$'').
  For $\sigma$ to appear in the co-transition formula, we also have
  $\sigma(i)\neq \pi(i)$, hence $i\in \{a,b\}$. Finally, for $(i,\pi(i))$
  to be just outside the dominant part, we need $\pi(c)<\pi(i) \implies c>i$.

  The case we need to rule out is $a=i, b>n$. Since $\pi \in S_n$, 
  we can't have $b \geq n+2$ ($n+1$ would be in the way).
  What remains is to rule out $b=n+1$. For each $c \in (i,n+1)$ to not be 
  in the way, we would need $\pi(c) > n+1$ (impossible since $\pi\in S_n$)
  or $\pi(c) < \pi(i)$. So $\pi(c)<\pi(i) \implies c>i \implies \pi(c)<\pi(i)$,
  setting up a correspondence between the $\pi(i)-1$ numbers $<\pi(i)$
  and the $n-i$ numbers $>i$. But then $i+\pi(i) = n+1$, contradicting 
  our choice of $(i,\pi(i))$. 
\end{proof}

\begin{proof}[Proof of this more general co-transition formula, for $P=G$] 
  For $\pi,i$ as in this more general co-transition formula, we have
  $$ \barX_\pi \cap \{M:m_{i\pi(i)} = 0 \}
  \ =\  \barX_\pi \cap \left\{M:
  m_{ab} = 0\ \forall (a,b) \text{ weakly NW of }(i,\pi(i)) \right\} $$
  since all those $(a,b)$ entries other than $(i,\pi(i))$ itself
  are already zero.

  The first description shows that the intersection is a hyperplane section
  (and nontrivial: $m_{i,\pi(i)} \not\equiv 0$ on $\barX_\pi$)
  of the irreducible $\barX_\pi$, so each component of the
  intersection is codimension $1$ in $\barX_\pi$. Moreover,
  $\left[\{M:m_{i\pi(i)} = 0 \} \cap \barX_\pi\right] = 
  \left[\{M:m_{i\pi(i)} = 0 \}\right]  [\barX_\pi] = (x_i - y_{\pi(i)}) G_\pi$.

  The benefit of the second description is that the two varieties
  being intersected are plainly $(B_-\times B_+)$-invariant.
  Hence that intersection is a union of $(B_-\times B_+)$-invariant
  subvarieties, each of which is necessarily a matrix Schubert variety
  $\barX_\rho$ by the Bruhat decomposition of $M_n(\CC)$.

  So far we know set-theoretically that the intersection is some
  union of $\barX_\rho \subseteq \{M:m_{i\pi(i)} = 0 \}$ (for, as yet,
  {\em partial} permutation matrices $\rho$) with 
  $\dim \barX_\rho = \dim \barX_\pi + 1$. 

  Hence $\rho(i)\neq \pi(i)$, with $\rho \gtrdot \pi$. 
  What remains is to show that every such $\rho \in S_n$ occurs,
  with multiplicity $1$, and
  that partial permutations $\rho$ (i.e. not in $S_n$) don't occur.
  Then we'll know that 
  $\left[\{M:m_{i\pi(i)} = 0 \} \cap \barX_\pi\right] = 
  \sum \left\{ 
    1\cdot [\barX_\rho]\ :\ \rho \in S_n, \ \rho \gtrdot \pi, 
    \ \rho(i) \neq \pi(i) \right\} $.

  Certainly the permutation matrix $\rho$ is in  
  $\{M:m_{i\pi(i)} = 0\}$ and $\barX_\pi$.
  If a partial permutation $\rho$ of corank $k$ were to give 
  a component, then upon stabilizing $\pi$ to $\pi^+ := \pi \oplus I_k$,
  the permutation matrix $\rho^+$ (chosen to have the same diagram as $\rho$)
  would give a component. But then $\rho^+ \gtrdot \pi^+$, and by the
  same argument as in lemma \ref{lem:stable} $\rho^+ \in S_n$, i.e. $k=0$.

  Finally, we need to show the multiplicity of the component $\barX_\rho$
  is $1$, i.e. the tangent space to $\{M:m_{i\pi(i)} = 0\} \cap \barX_\pi$ 
  at the point $\rho$ is just $T_\rho \barX_\rho$.
  This was lemma \ref{lem:Tcover}.
\end{proof}

\section{An inductive pipe dream formula}\label{sec:inductive}

The formula defining $C_\pi$ as a sum over pipe dreams,
and the co-transition formula writing $(x_i-y_{\pi(i)})C_\pi$ as a sum of
other $C_{\rho}$s, have a common generalization. We include it here,
though it's not actually required for the main theorems.

Take $\pi \in S_n$, and let $\lambda$ be an English partition fitting in the
strict Northwest triangle (i.e. $\lambda_i \leq n-i$ for $i \in [1,n]$).
Define a \defn{partial pipe dream} for the pair $(\pi, \lambda)$ to be
\begin{itemize}
\item a tiling of $\lambda$ with the two tiles as usual, and 
\item a chord diagram in the complement $\square/\lambda$ of $\lambda$
  in the square $\square$, whose $n$ chords have endpoints at the
  centers of the North and West edges of $\square/\lambda$,  
  considered up to isotopy and braid moves,
\end{itemize}
such that
\begin{itemize}
\item each chord has positive slope, hence connects a West end to a
  North end to its Northeast, and
\item the combination of the pipes in $\lambda$ and chords in
  $\square/\lambda$ connect $1\ldots n$ on the North side to
  $\pi(1)\ldots \pi(n)$ down the West side.
\end{itemize}
Some examples are given in figure \ref{fig:partialPDex}.
As the pictures suggest, one can consider the $\lambda$ region as the
``crystalline'' part of the diagram, and the $\square/\lambda$ complement
as the ``molten'' region. One uses the co-transition formula to
freeze more, increasing $\lambda$.

Associate a second permutation $\rho$ to a partial pipe dream $D$
for $\pi,\lambda$ by
$$ \rho(D) :=
\left(\prod_{{}^i \textelbow_j \ \in D} (\pi(\text{row}) \leftrightarrow j) \right)
\ \circ\ \pi
\qquad\qquad \text{product ordered NW to SE} $$

\begin{figure}[h]\label{fig:partialPDex}
  \centerline{   \epsfig{file=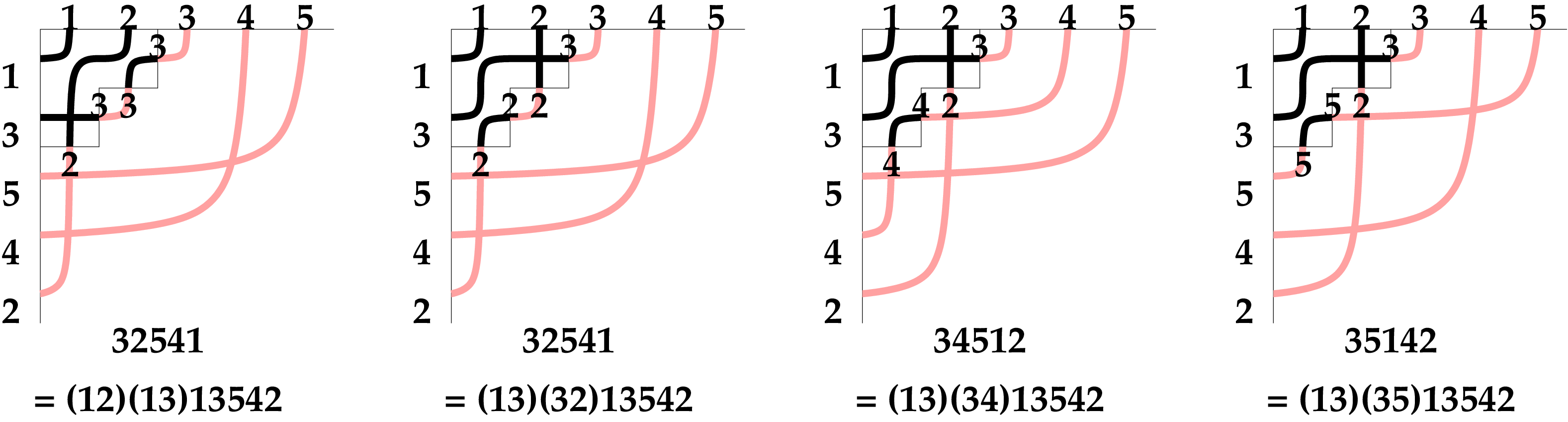,width=6in} }
  \caption{ The four partial pipe dreams $D$ for
    $(\pi = 13542$, $\lambda = 2+1)$,
    with each one's $\rho(D)$ written below it.
    Note that the tiles alone do not characterize the partial pipe dream;
    one must number the pipes. (When $\lambda$ is the full staircase every
    pipe connects to the North and West boundaries, determining its number.)
  }
\end{figure}

\begin{Theorem}\label{thm:partialPD}
  Fix $\pi\in S_n$ and $\lambda$ a partition with
  $\lambda_i \leq n-i, i \in [1,n]$. Then
  $$ C_\pi
  = \sum_D \frac{1}{\prod_{\textelbow\in D} (x_{row}-y_{col})}
  C_{\rho(D)}
  $$
  where $D$ varies over the partial pipe dreams for $(\pi,\lambda)$.
  The dominant part of the Rothe diagram of each $\rho(D)$ contains $\lambda$,
  so each summand is polynomial.
\end{Theorem}

\begin{proof}[Proof sketch.]
  Call two pipe dreams for $\pi$ \defn{$\lambda$-equivalent}
  if they agree (in tiles and labels on pipes) inside $\lambda$.
  There is an obvious map from equivalence classes
  to partial pipe dreams, which we baldly assert to be bijective.
  If we take the pipe dreams in an equivalence class $D$ and replace
  all the $\textelbow$s in $\lambda$ with $\textcross$s,
  we further assert that we get exactly the pipe dreams for $\rho(D)$.
  The result follows.
\end{proof}

If we take $\lambda$ to be the dominant part of $\pi$, then there
is only one partial pipe dream $D$ for $(\pi,\lambda)$, where
$\lambda$ is solid $\textcross$s, and theorem \ref{thm:partialPD}
says $C_\pi = C_\pi$ (since $\rho(D)=\pi$). If we take $\lambda$ to
have {\em one} more square at $(i,\pi(i))$, then the only freedom
in $D$ is the choice $\textelbow_j$ of pipe label $j$ on that square,
and theorem \ref{thm:partialPD} becomes the generalized co-transition formula
from \S\ref{sec:matrix}. Finally, if we take $\lambda$
to be the full staircase $(n-1)+(n-2)+\ldots+1$, then every $D$ has
$\rho(D) = w_0$, and theorem \ref{thm:partialPD} recovers
the definition of $C_\pi$ as a sum over pipe dreams.

\section{Transition vs. co-transition}\label{sec:trans}

In \cite{KM} the Fulton determinants defining $\barX_\pi$ were shown 
to be a Gr\"obner basis for antidiagonal term orders $<$,
and the components of $init_<\ \barX_\pi$ to be in obvious correspondence
with $\pi$'s pipe dreams.
There are four natural sources of antidiagonal term orders:
\begin{enumerate}
\item lexicographic, where the matrix entries are ordered from NE to SW
  (more precisely, by some linear extension of that partial order)
\item lexicographic, where the matrix entries are ordered from SW to NE
\item reverse lexicographic, where the matrix entries are ordered from NW to SE
\item reverse lexicographic, where the matrix entries are ordered from SE to NW.
\end{enumerate}
Slicing $\barX_\pi$ with the hyperplane $m_{i,\pi(i)}=0$ is a way of doing
the first nontrivial step of the third kind of Gr\"obner degeneration, 
and hence, will a priori be compatible with the pipe dream combinatorics. 
(It is from there that the co-transition formula, 
and \S \ref{sec:pipedream}, were reverse-engineered.
Stated more bluntly: after this insight, producing the rest of the paper was
essentially an exercise.)

Define the \defn{co-dominant part outside} $\pi$'s Rothe diagram 
as the set of matrix entries $(a,b)$ such that no 
Fulton determinant defining $\barX_\pi$ involves $m_{ab}$. 
This is always connected to the SE corner of the square. Its complement
is the boxes NW of some diagram box, or equivalently NW of some essential box.
The $(i,j)$ in Lascoux' transition formula was picked to be just outside
the co-dominant part outside $\pi$'s Rothe diagram.
See \cite{trunc} for this view of the transition formula.

In unpublished work, Alex Yong and I gave a Gr\"obner-degeneration-based
proof of Lascoux' transition formula, based on one step of a {\em lex} order
from SE to NW (so, not one of the orders above compatible with pipe dreams).
For this reason, one might expect it to be very difficult to connect the
pipe dream formula to the transition formula, requiring 
``Little bumping algorithms'' and the like (see \cite{BHY}),
and essentially impossible if one wants to include the $\ul y$ variables.
Indeed, it should be about as difficult as giving a bijective proof that
two unimodular triangulations of a polytope should have the same number
of simplices. (See \cite{EM} where polytopes arise from some matrix Schubert
varieties, and this becomes more than an analogy.)

Recall the \defn{conormal variety} $CX$ of a closed subvariety $X \subseteq V$
of a vector space:
$$ CX := \overline{ \left\{ (x, f) \in V\times V^*\ :\ 
    x\in X \text{ a smooth point, } \vec v \perp T_x X \right\} }
\quad \subseteq V \times V^* $$
Use the trace form to identify $M_n(\CC)^*$ with $M_n(\CC)$, and
call two matrix Schubert varieties $\barX_\pi$, $\barX_\rho$ 
\defn{projective dual} if $C\barX_\pi \subseteq M_n(\CC) \times M_n(\CC)$
becomes $C\barX_\rho$ upon switching the two $M_n(\CC)$ factors and rotating
both matrices by $180^\circ$. (This is essentially the statement that
the projective varieties $\PP(\barX_\pi), \PP(\barX_\rho)$ are 
projective dual in the 19th-century sense; our reference is
\cite{Tevelev}.) It is a fun exercise to determine $\rho$ from $\pi$;
note that at least one of the two must be partial, not a permutation.

If $\barX_\pi$ and $\barX_\rho$ are projectively dual, then the
dominant part of $\pi$'s diagram is the $180^\circ$ rotation of the
co-dominant part outside $\rho$'s diagram -- projective duality swaps
zeroed-out coordinates with free coordinates.

Projective duality also exchanges lex term orders with revlex term orders.
So finally, in this sense, the co-transition formula is related to the
transition formula by projective duality. (The relation would be exact 
were to consider Gr\"obner degenerations of the conormal varieties, 
rather than of the matrix Schubert varieties themselves; since we only see
the components in one $M_n(\CC)$ or the other the relation is more
of an analogy.)

The reader may wonder, since the lex-from-NE term order was useful
(this is effectively the approach in \cite{transformationgroups}) and
the revlex-from-NW term order was useful (in \S\ref{sec:matrix}),
why are the other two (at $180^\circ$ from these) left out?
The $180^\circ$ symmetry is achieved if we refine the matrix Schubert
variety stratification on $M_n(\CC)$ to the pullback of the
{\em positroid stratification} on $Gr(n;\ \CC^{2n})$ along the inclusion
$graph:\, M_n(\CC) \into Gr(n;\ \CC^{2n})$ regarding $M_n(\CC)$ as
the big cell. 

In \cite{LLS} was introduced an alternative family $\{C'_\pi\}$
of ``bumpless pipe dream'' polynomials, and a proof that they match
the double Schubert polynomials. The bijection from \S\ref{sec:pipedream}
deriving the co-transition formula for the pipe dream polynomials $\{C_\pi\}$
has a tightly analogous bijective proof of the transition formula
for the $\{C'_\pi\}$, in the recent preprint \cite[\S 5]{Weigandt}.

\section{Grothendieck polynomials, 
  nonreduced pipe dreams, \\ 
  and equivariant $K$-classes}\label{sec:K}

All three families of polynomials $A,C,G$ have extensions to 
inhomogeneous Laurent polynomials $A',C',G'$ in
$\ZZ[\exp(\pm x_1),\exp(\pm x_2),\ldots,\exp(\pm y_1),\exp(\pm y_2),\ldots]$:
\begin{enumerate}
\item {\em Double Grothendieck polynomials $A'_\pi$.} These satisfy
  recurrence relations based on isobaric Demazure operators.
\item {\em Nonreduced pipe dream polynomials $C'_\pi$.} These allow pipes to
  cross twice. To read a permutation off of a (nonreduced) pipe dream,
  one follows the pipes, ignoring the second (and later) crossings
  of any two pipes.
\item {\em Equivariant $K$-classes of matrix Schubert varieties $G'_\pi$.}
  The subvariety $\barX_\pi \subseteq M_n(\CC)$ defines a class in
  $(T\times T)$-equivariant $K$-theory of $M_n(\CC)$.
\end{enumerate}
Betraying our predilection towards geometry, we call each the
``$K$-theoretic version'' of the unprimed family, with the original
being the ``cohomological''.

Each $K$-theoretic family satisfies the same new base case

\begin{Lemma*}[$K$-theoretic base case]
  For each family $P'$ we have $P'_{w_0^n} = 
  \displaystyle\prod_{i,j \in [n],\ i+j\leq n} (1 - \exp(y_i - x_j))$.
\end{Lemma*}

\noindent and the recurrence

\begin{Lemma*}[the $K$-theoretic co-transition formula]
  Let $\pi,i,\{\rho\}$ be as in the cohomological co-transition formula.
  Let $S$ vary over the nonempty subsets of the set of such $\rho$.
  Then
  $$ (1 - \exp(y_{\pi(i)} - x_i)) \, P'_\pi 
  = \sum_S (-1)^{\#S-1}\, P'_{l.u.b.(S)} $$
  where $l.u.b.(S)$ is the (unique) least upper bound of $S$ in
  Bruhat order, automatically of length $\ell(\pi) + \#S$.
\end{Lemma*}

Intriguingly, this ``boolean lattice inside Bruhat order'' phenomenon shows up 
in the $K$-theoretic transition formula \cite{LascouxTransition} as well.

We won't prove these two for $A',C',G'$, but comment on the changes 
necessary from the cohomological proofs. (Of course, it is already
known that $A' = C' = G'$, see e.g. \cite{KM}, so it suffices to prove
these results for, say, just $G'$.) The bijection in $C'$, placing
a $\textcross$ at $(i,\pi(i))$ where there was always a $\textelbow$,
is the same. For the $G'$ co-transition formula one needs to know
that the intersection $\barX_\pi \cap \left\{M:
  m_{ab} = 0\ \forall (a,b) \text{ weakly NW of }(i,\pi(i)) \right\}$
is reduced, and that each intersection $\cap_S \barX_\rho = \barX_{l.u.b.(S)}$
is likewise reduced. The swiftest way to confirm this is to observe that
there is a Frobenius splitting on the space of matrices 
(over each $\mathbb F_p$, rather than $\CC$),
with respect to which each $\barX_\pi$ is compatibly split;
as at the end of \S \ref{sec:trans}, one can infer this from the
compatible splitting of the positroid varieties in the Grassmannian
\cite{KLS}.

\end{document}